\documentclass[12pt]{amsart}
\usepackage{amsmath,amssymb,amsthm,amsfonts}
\newtheorem{theorem}{Theorem}[section]
\theoremstyle{definition}

\newtheorem{proposition}[theorem]{Proposition}

\newtheorem{conjecture}[theorem]{Conjecture}

\theoremstyle{remark}

\numberwithin{equation}{section}
\begin{document}
\title[A note on the Navarro conjecture for alternating groups]{A note on the Navarro conjecture for alternating groups with abelian defect}
\author{Rishi Nath}
\address{York College/City University of New York}
\email{rnath@york.cuny.edu} 
%\thanks{A preliminary version of these results appear in the PhD thesis of the author \cite{Na2}.}
%\keywords{Navarro conjecture, McKay conjecture, alternating groups, Galois automorphisms}

%\date{April 11, 2008}
\begin{abstract}
G.Navarro proposed (in \cite{Nav}) a refinement of the unsolved McKay conjecture
involving certain Galois automorphisms. The author verified this new conjecture for the alternating groups $A(\Pi)$ when $p=2$ (see \cite{Na2}).  For odd primes $p$ the conjecture is more difficult to study due the complexities in the $p$-local character theory. We consider the principal blocks of $A(\Pi)$ with an abelian defect group when $p$ is odd: in this case the Navarro conjecture holds for $p$-singular characters.\\
%{\center \bf keywords:} Navarro conjecture, McKay conjecture, alternating groups, Galois autmorphisms
\end{abstract}
\maketitle
\section{McKay and Navarro conjectures}
\noindent
\\
Let $G$ be a finite group, $|G|=n$, $p$ be a prime
dividing $n$, $D$ a Sylow $p$-group of $G$, and $N_{G}(D)$
the normalizer of $D$ in $G$. Let $Irr (G)$ denote the
irreducible characters of $G$, and $Irr_{p'}(G)$ the subset of
characters whose degree is relatively prime to $p$. The
following is a well-known conjecture.
\begin{conjecture}(McKay, \cite{Alp})\label{McKayCon}
\[
|Irr_{p'}(G)|=|Irr_{p'}(N_{G}(D))|.
\]
\end{conjecture}
Recently G. Navarro
strengthened the McKay conjecture
in the following way. All irreducible
complex characters of $G$ are afforded by a representation with values in the $n$th
cyclotomic field ${\mathbb Q}_n/{\mathbb Q}$ (Lemma 2.15, \cite{I}).
Then the Galois group $\mathcal{G}=\text{Gal}({\mathbb
Q}_n/{\mathbb Q})$ permutes the elements of $Irr (G).$ We denote the action of $\sigma$ on $\chi\in Irr (G)$ by $\chi^{\sigma}.$ Then
$\chi\in Irr (G)$ is $\sigma$-{\it fixed} if its values are fixed by $\sigma$, that is, $\chi^{\sigma}=\chi.$
Let $e$ be a nonnegative integer and consider
$\sigma_{e}\in \mathcal{G}$ where $\sigma_e(\xi)=\xi^{p^e}$ for
all $p'$-roots of unity $\xi$. Define $\mathcal{N}$ to be the subset
of $\mathcal{G}$ consisting of all such $\sigma_e.$
Let
$Irr^{\sigma}_{p'} (G)$ and $Irr^{\sigma}_{p'} (N_{G}(D))$ be the
subsets of $Irr_{p'}(G)$ and $Irr_{p'}(N_{G}(D))$ respectively
fixed by $\sigma\in \mathcal{N}$.
\begin{conjecture}(Navarro, \cite{Nav}) \label{NavConj}
Let $\sigma\in\mathcal{N}.$ Then
\[
|Irr^{\sigma}_{p'} (G)|=|Irr^{\sigma}_{p'} (N_{G}(D))|.
\]
\end{conjecture}
The Navarro conjecture follows from the existence of a bijection $\phi$ from $Irr_{p'}(G)$ to $Irr_{p'}(N_{G}(D))$
that commutes with $\mathcal{N}$. That is, $\phi(\chi^{\sigma})=\phi(\chi)^{\sigma}$ for all $\sigma\in\mathcal{N}$
and $\chi\in Irr_{p'}(G)$. The author verified in \cite{Na1} that the Navarro conjecture holds for the alternating groups $A(\Pi)$
when $p=2$. The verification when $p$ is odd is more complicated since little is known
about values of $Irr_{p'}(N_{A(\Pi)}(D))$.
However in the special case that $A(\Pi)$ has an abelian defect group (equivalently $|\Pi|=n_0+wp$ with $w<p$)
this paper verifies that the Navarro conjecture holds for the $p$-singular characters of the principal block.  The proof relies on results of P. Fong and M. Harris (see $\S$4, \cite{F-H}) on the irrationalities of the $p$-singular characters of $N_{A(\Pi)}(D)$.
%%%%%%%%%%%%%%%%%%%%%%%
\section{A local-global bijection}
\subsection{$p'$-splitting characters of $G$}
\noindent
\\
Let $n\in\mathbb{N}.$ A {\it partition} $\lambda$ of $n$ is a non-increasing integer sequence
$(a_1,\cdots,a_m)$
satisfying $a_i \geq\cdots\geq a_m$ and $\sum_i a_i=n.$
Then the {\it Young diagram} of $\lambda$ is
$n$ nodes placed in rows such that the $i$th row of $\lambda$ consists
of $a_i$ nodes. The $(i,j)$-{\it node} of $\lambda$ lies in the $i$th row and $j$th column of the Young diagram. The $(i,j)$-{\it hook} $h^{\lambda}_{ij}$ of [$\lambda$] and consists of the $(i,j)$-node (or {\it corner} of $h^{\lambda}_{ij}$),
all nodes in the same row and to the right of the corner, and all nodes in the same column and below the corner.
The column-lengths of $[\lambda]$ form the {\it conjugate}
partition $\lambda^*$ of $n$. Partitions where
$\lambda=\lambda^{*}$ are {\it self-conjugate}. Let $\lambda=\lambda^*$ and
$\delta(\lambda)=\{\delta_{jj}\}$ be the set of {\it diagonal hooks} of
$\lambda$ i.e. $\delta_{jj}=h_{jj},$ which are necessarily odd. When there is no ambiguity we write $h^{\lambda}_{ij}=h_{ij}$.

Every $\lambda$ is expressed uniquely in terms of its
{\it p-core} $\lambda^0$ and {\it p-quotient} $(\lambda_{0},\lambda_{2},\cdots,\lambda_{p-1})$. The
{\it p-core} $\lambda^0$ is the unique partition that results when all possible
hooks of length $p$ are removed from $\lambda$.  The $p$-quotient $\langle\lambda\rangle$ is a $p$-tuple of (sub-)partitions which encode the $p$-hooks of
$\lambda$.

Henceforth, let $\Pi$ be a set of size $n$ and $G=S(\Pi)$ and $G^+=A(\Pi)$
be respectively the symmetric and alternating groups on $\Pi$.
The elements of $Irr(G)$ are labeled by partitions $\{\lambda\vdash n\}.$ Then $Irr(G^+)$ is obtained from $Irr(G)$
by restriction.  If $\alpha$ is an irreducible character for some finite group $J$, and $K$ is a subgroup of $J$, the notation $\alpha |_{K}$ indicates restriction of the subgroup $K.$
\begin{theorem}\label{thmFrob1}
The irreducible characters of $G^+$ arise from those of $G$
in two ways. If $\lambda\neq\lambda^{*}$ then $\chi_{\lambda}|_{G^+}=\chi_{\lambda^{*}}|_{G^+}$ is in $Irr(G^+)$. If $\lambda=\lambda^{*}$ then $\chi_{\lambda}|_{G^+}$ splits into two
conjugate characters $\chi_{\lambda}^{+}$ and $\chi_{\lambda}^{-}$
in $Irr(G^+)$.
\end{theorem}
The conjugacy classes $\kappa$ of $S(\Pi)$ are labeled by
cycle-types of permutations of $n$.  If $\lambda=\lambda^*$ we let
$\kappa_{\delta(\lambda)}$ be the conjugacy class determined by the
cycle-type of $(\delta_{11},\cdots,\delta_{dd})$. Then
$\kappa_{\delta(\lambda)}$ splits into $\kappa_{\delta(\lambda),+}$
and $\kappa_{\delta(\lambda),-}$ when viewed as a class of $G^+.$
Let $Irr^*(G)$ be the set of {\it splitting characters}, i.e. those
that split into two conjugate characters when restricted to $G^+.$
The following is a classical result of Frobenius (see e.g. Theorem
(4A), \cite{F-H}).
\begin{theorem}\label{Frob}
Suppose $\chi_{\lambda}$ is an irreducible character of $G$ which
splits on $G^+.$ Let $g\in G^+.$ Then
$(\chi_{\lambda,+}-\chi_{\lambda,-})(g)\neq 0$ if and only if $g$ is
in $\kappa_{\delta(\lambda)}.$ Moreover, $\chi_{\lambda,\pm}$ and
$\kappa_{\delta(\lambda),\pm}$ may be labeled so that
\[
\begin{array}{ccc}
\chi_{\lambda}^{\pm}(g)&=&\frac{1}{2}[\epsilon_{\lambda}+
\sqrt{\epsilon_{\lambda}\prod_{j}\delta_{jj}}]\text{\;\; if
$g\in\kappa_{\delta(\lambda),\pm}$}\\
\chi_{\lambda}^{\pm}(g)&=&\frac{1}{2}[\epsilon_{\lambda}-
\sqrt{\epsilon_{\lambda}\prod_{j}\delta_{jj}}]\text{\;\; if
$g\in\kappa_{\delta(\lambda),\mp}$}
\end{array}
\]
where $\epsilon_{\lambda}=(-1)^{\frac{n-d}{2}}$.
\end{theorem}
By extension, $Irr^*(G^+)$ is the set of (pairs) of characters that
arise from restricting elements of $Irr^*(G).$ Suppose $n=w p,$ where $w<p$. By a condition of Macdonald (see \cite{Mac}), the elements of $Irr_{p'}(G)$ are labeled by partitions for whom $\sum |\lambda_{\gamma}|= \omega.$ Then the $p'$-splitting characters are labeled by self-conjugate partitions that satisify the Macdonald condition.
%%%%%%%%%%%%%
%%%%%%%%%%%%%%%%%%
\subsection{$p'$-splitting characters of $H$}
\noindent
\\
Let $B$ be a $p$-block of $G$ the defect group $D$ and $b$ the
$p$-block of $N_G(D)$ which is the Brauer correspondent of B. Let
$\nu$ be the exponential valuation of ${\mathbb Z}$  associated with
$p$ normalized so $\nu(p)=1.$ The height of the $\chi$ in $B$ is the
nonnegative integer $h(\chi)$ such that
$\nu(\chi(1))=\nu(|G|)-\nu(|D|)+h(\chi).$ The height  of $\xi$ in
$b$ is the nonnegative integer $h(\xi)$ such that
$\nu(\xi)=\nu(|N_{G}(D)|)-\nu(|D|)+h(\xi).$ Let $M(B)$ and $M(b)$ be
the characters of $B$ and $b$ of height zero. By the Nakayama
conjecture a $p$-block $B$ of $G$ is parametrized by a $p$-core
$\lambda^0$ so $\chi_{\mu}\in B$ if and only if $\lambda^0=\mu^0.$
In particular, $n=n_0+wp$ where $n_0=|\lambda_0|$. We suppose that
$B$ has abelian defect group $D$ or equivalently $w<p.$ Thus
$\Pi=\Pi_0\cup \Pi_1$ is the disjoint union of sets $\Pi_0$ and
$\Pi_1$ of cardinality $n_0$ and $wp$.  We may suppose
$\Pi_1=\Gamma\times \Omega$ where $\Gamma=\{1,2,\cdots,p\}$ and
$\Omega$ is a set of $w$ elements. Let $X=S(\Gamma)$ and $Y=N_X(P)$
where $P$ is a fixed Sylow $p$-subgroup of $X$. Note that the when
$B$ is a Sylow subgroup the $p'$-irreducible characters agree with
the height zero characters.

We take $D$ as the Sylow $p$-subgroup $P^{\Omega}$ of $S(\Pi_1)$ and set $H=N_G(D)$ so that $H=H_0\times H_1$
with $H_0=S(\Pi_0)$ and $H_1=Y\wr S(\Omega).$ The Brauer correspondent $b$ of $B$ in $H$
has the form $b_0\times b_1$ where $b_0$ is the block of defect 0 of $H_0$ parametrized by $\lambda^0$ and $b_1$ is the principal block of $H_1$.

Let $\lambda$ be a partition of $n$ with $p$-core $\lambda^0$ and $p$-quotient $\langle\lambda\rangle=(\lambda_0,\cdots,\lambda_{p-1})$ normalized as follows: if $\mu=\lambda^*$
then $\lambda_i=(\mu_{p-i-1})^*$. Let $p^*=\frac{p-1}{2}.$ Then $\lambda=\lambda^*$ implies $\lambda_{p^*}=\lambda^*_{p^*}$. Let $Y^{\vee}=\{\xi_{\gamma}:0\leq \gamma\leq p-1\}.$ The characters in $H^{\vee}$ have the form $\chi_{\tau}\times \psi_{\Lambda}$ where $\tau$ is a $p$-core partition and
$\chi_{\tau} \in Irr(H_0)$ and $\psi_{\Lambda}\in Irr(H_1)$ and $\Lambda$ is a mapping
\[
Y^{\vee}\longrightarrow\{\textit{Partitions}\},\;\; \xi_{\gamma}\mapsto \mu_{\gamma},
\]
such that $\sum_{\gamma}|\mu_{\gamma}|=w.$ We also represent
$\Lambda$ by the $p$-tuple $(\mu_1,\cdots,\mu_p).$ Then $M(B)$ and
$M(b)$ are in bijection via $f:{\chi_{\lambda}}\mapsto
\chi_{\lambda^0}\times\psi_{\langle\lambda\rangle}$ (see \cite{Fo}
for details). Hence $Irr_{p'}(G)$ and $Irr_{p'}(H)$ are in bijection
via $f=\cup_B f_B.$ There is an induced bijection $f^+$ between
$Irr_{p'}(G^+)$ and $Irr_{p'}(N_{G^+}(D))$. Let $sgn_H=sgn_{G}|_{H}$
and $sgn_Y=sgn_X|_Y$.  If $(f,\sigma)$ is an element of $H=Y\wr
S(\Omega)$ with $f\in S(\Omega)$ and $f\in Y^{\Omega}$ and
$\sigma\in S(\Omega)$, then
\[
sgn_H(f,\sigma)=sgn_{S(\Omega)}(\sigma)\prod_{i\in \Omega}sgn_{Y}(f(i)).
\]
Let $H^+=N_{G^+}(D)$. Then $\Lambda$ is a {\it splitting
mapping} of $H$ if $\psi_{\Lambda}$ {\it splitting character} of $H$ i.e. $(\psi_{\Lambda})|_{H^+}=\psi_{\Lambda,+}-\psi_{\Lambda,-}$
where $\psi_{\Lambda,\pm}\in (H^+)^{\vee}$. Let $^*$ be the duality $\Lambda\mapsto \Lambda^*$ where $\Lambda^*:\xi_{\gamma}\mapsto (\lambda_{p-1-\gamma})^*.$
The following is Proposition (4D) in \cite{F-H}.
\begin{proposition}\label{pwsplits} Let $\psi_{\Lambda}\in Irr(H)$. Then $sgn_H \psi_{\Lambda}=\psi_{\Lambda^*}$. In particular, $\psi_{\Lambda}$ is a splitting character
if and only if $\Lambda=\Lambda^*.$
\end{proposition}
Proposition \ref{pwsplits} implies that map $f^+$ induced by $f$ remains a bijection on splitting characters (and $p'$-splitting characters). That is, $Irr^*_{p'}(G^+)$ is in bijection with $Irr^*_{p'}(H^+).$  In particular, if $\lambda\neq\lambda^*$ then $\chi_{\lambda} |_{G^+}=\chi_{\lambda}^* |_{G^+}$ is mapped to $\psi_{\Lambda} |_{H^+}=\psi_{\Lambda^*} |_{H^+}$ and if $\lambda=\lambda^*$ then $\chi^{\pm}_{\lambda}$ maps to $\psi^{\pm}_{\Lambda}$.

%%%
\section{Values of $p$-singular characters}
\noindent
We say $\lambda$ is {\it p-singular} if $\lambda_{p^*}\neq \emptyset$ and $\lambda_i=\emptyset$ for all $i\in\{0,\cdots,p-1\}-p^*$.
Then $\chi_{\lambda}\in Irr_{p'}(G)$ is $p$-singular if $\lambda$ is. The notation $Irr_{p', sing}(G)$ denotes the $p$-singular $p'$-characters and $Irr_{p',sing}(G^+)$ is the restrictions to $G^+$. Then $Irr_{p', sing}(H)$ and $Irr_{p',sing}(H^+)$ are defined analogously. It is immediate from the definition of $f^+$ that $Irr_{p', sing}(G^+)$ and $Irr^*_{p',sing}(H^+)$ are in bijection. We show that $f^+$ commutes with the action of $\sigma\in \mathcal{N}$ on $p$-singular $p'$-characters by describing explicitly the relevant irrational character values.

In \cite{Na1}, the author describes how to obtain the set of diagonal hooks $\delta(\lambda)$ of a symmetric partition $\lambda=\lambda^*$ given just the $p$-core $\lambda^0$ and the $p$-quotient $\langle \lambda \rangle.$ The following special case (Theorem 4.3 in \cite{Na1}) is relevant to the goals of this paper.
\begin{theorem} \label{diagonal_D}
Suppose $\lambda^0$ is empty and $(\emptyset,\cdots,\lambda_{p^*},\cdots,\emptyset)$ such that $\lambda_{p^*}=(\lambda_{p^*})^*$ and $\delta(\lambda_{p^*})$=$(\delta'_{11},\cdots,\delta'_{dd}).$ Then $\delta(\lambda)$=$(\delta'_{11}p,\cdots,\delta'_{dd}p).$
\end{theorem}
A conjugacy class $C$ of $H$ is a {\it splitting class} if $C\subseteq H^+$ and $C=C_-\cup C_+$ is the union of two conjugacy classes of $H^+$.  There is a bijection between splitting mappings $\Lambda$ and splitting classes
$C_{\Lambda}$ of $H$ (see pg.3491, \cite{F-H}). The following is Proposition (4F) in \cite{F-H}.
\begin{theorem} \label{psingular-irrational} Let $|\Pi|=wp.$ Suppose $\Lambda$ is a splitting mapping of $N_{S(\Pi)}(D)$ that equals its $p$-singular part
i.e. $\Lambda=$ $(\emptyset,\cdots,\lambda_{p^*},\cdots,\emptyset)$.
Let $(f,\sigma)\in N_{A(\Pi)}(D)^+$. Then
$(\psi_{\Lambda,+}-\psi_{\Lambda,-})(f,\sigma)$$\neq 0$ if and only
if $(f,\sigma)\in C_{\Lambda}.$ Moreover, $\psi_{\Lambda,\pm}$ and
$C_{\Lambda,\pm}$ may be labeled so that
\[
(\psi_{\Lambda,+}-\psi_{\Lambda,-})(f,\sigma)\\=\pm(\sqrt{\epsilon_p
p})^d \sqrt{\epsilon_{\lambda_{p^*}\prod_j\eta_{jj}}}
\]
for $(f,\sigma)\in C_{\Lambda,\pm},$ where,
$\epsilon_{\lambda_{p^*}}=(-1)^{\frac{p-1}{2}}$, $d$ is the number
of diagonal nodes in $\lambda_{p^*}$ and
$\delta(\lambda_{p^*})$$=(\eta_{11},\cdots,\eta_{dd}).$
\end{theorem}
Suppose $\sigma\in Gal(\mathbb{Q}_{|G^+|}/\mathbb{Q})$ is such that $\sigma(\xi)=\xi^{p^e}$ for some $e\in \mathbb{Z}^+$ and $\xi$ is a $p'$-root of unity. We define $Irr_{p'}(B_1)$ and $Irr_{p'}(b_1)$ to be the $p'$-characters of the principal block $B_1$ of $A(\Pi)$ and its Brauer correspondent $b_1$ and $Irr_{p', sing}(B_1)$ and $Irr_{p',sing}(b_1)$ are defined by extension.
\begin{theorem} Let $A(\Pi)$ be the alternating group on $\Pi$ and $p$ is an odd prime such that $A(\Pi)$ has an abelian defect group. Let $\sigma\in \mathcal{N}.$ Let $B_1$ be the principal block of $A(\Pi)$, $\chi\in Irr_{p'}(B_1)$ and $b_1$ its Brauer correspondent. Then the restriction of $f^+$ is a bijection between $Irr_{p',sing}(B_1)$ and $Irr_{p',sing}(b_1)$  that commutes with $\sigma.$ That is, $f^+(\chi)^{\sigma}=f^+(\chi^{\sigma})$.
\end{theorem}
\begin{proof}
Since $A(\Pi)$ has abelian defect, and we are considering only the principal block, we can assume $|\Pi|=wp.$ By the discussion above, we consider two cases.
\begin{enumerate}
\item Suppose $\lambda\neq\lambda^*.$ Then the restrictions $\chi_{\lambda} |_{G^+}=\chi_{\lambda^*} |_{G^+}$ are in bijection with $\psi_{\Lambda^*} |_{H^+}=\psi_{\Lambda^*} |_{H^+}$.  Since the values of $\chi_{\lambda}$ are all rational, $\chi_{\lambda } |_*$ is $\sigma$-fixed.  Since $N_{G^+}(X)=Y\wr S(\Omega)$ where $|\Omega|=p$, $\psi_{\Lambda} |_{H^+}$ is also $\sigma$-fixed.
\item Suppose $\lambda=\lambda^*.$  Upon restriction, the pair $\chi^{\pm}$ is in bijection with the pair $\psi^{\pm}_{\Lambda}$ via $\bar{f}$ .  It remains to show that the values of $\chi^{\pm}_{\lambda}$ and $\psi^{\pm}_{\Lambda}$ on the splitting classes $\kappa^{\pm}_{\delta(\lambda)}$ and $C^{\pm}_{\Lambda}$ are both exchanged or fixed by $\sigma$.
By Theorem \ref{diagonal_D}, Theorem \ref{Frob},  and Theorem \ref{psingular-irrational}, $\sqrt{\eta_jp^d}=\sqrt{\delta_j}$. Since $p$ is odd, $(wp-d)\equiv(p-1+w-d)\pmod{2}$,
so $\epsilon_{\lambda_{p^*}}\cdot \epsilon_p =\epsilon_{\lambda}$. This completes the proof.
\end{enumerate}
\end{proof}
%%%%%%%%%%%%%}
%%%%%%%%%%%%%
\noindent {\bf Acknowledgements.} The author is indebted
to Paul Fong for his guidance and suggestions. This research was partially supported from a grant by
PSC-CUNY.

\end{document}